\documentclass{amsproc}

\numberwithin{equation}{section}

\usepackage{amsmath,amsfonts,amssymb,mathtools}
\usepackage{mathrsfs}

\newtheorem{global-theorem}{Theorem}
\newtheorem{theorem}{Theorem}[section]
\newtheorem{lemma}[theorem]{Lemma}

\newtheorem{corollary}[theorem]{Corollary}

\newtheorem{definition}[theorem]{Definition}
\newtheorem{proposition}[theorem]{Proposition}
\newtheorem{remark}[theorem]{Remark}

\newtheorem{principle}[theorem]{Principle}

\newcommand{\rr}{{\mathbb R}}

\newcommand{\four}{{\bf 4}}
\newcommand{\six}{{\bf 6}}
\newcommand{\eight}{{\bf 8}}

\newcommand{\Znew}{Z^{\rm new}}

\newcommand{\dxo}{\langle 0 \rangle}
\newcommand{\dxa}{\langle 1 \rangle}
\newcommand{\dxb}{\langle 2 \rangle}
\newcommand{\dxc}{\langle 3 \rangle}
\newcommand{\dxd}{\langle 4 \rangle}
\newcommand{\dxma}{\langle -1 \rangle}
\newcommand{\dxmb}{\langle -2 \rangle}
\newcommand{\dxmc}{\langle -3 \rangle}
\newcommand{\dxmd}{\langle -4 \rangle}

\begin{document}

\title{Reduction for $SL(3)$ pre-buildings}

\author{Ludmil Katzarkov}
\address{University of Miami, Department of Mathematics, 1365 Memorial Drive, Coral Gables, FL 33146}

\author{Pranav Pandit}
\address{Universit\"at Wien, Fakult\"at 
f\"ur Mathematik, Oskar-Morgenstern-Platz 1, 1090 Wien, Austria}

\author{Carlos Simpson}
\address{Universit\'e C\^ote d'Azur, CNRS, LJAD, France}

\thanks{The authors wish to express their gratitude to D. Auroux, F. Haiden and Maxim Kontsevich for their help.
 We were partially  supported by
 by Simons research grant, NSF DMS 150908, ERC Gemis, DMS-
1265230, DMS-1201475, OISE-1242272, PASI,  
ANR 933R03/13ANR002SRAR, Simons collaborative Grant - HMS,
Russian Megagrant HMS and Automorphic Forms. The third author would like to thank the University of Miami for hospitality during the completion of this work.}




\begin{abstract}
Given an $SL(3)$ spectral curve over a simply connected Riemann surface, we describe in detail the
reduction steps necessary to construct the core of a 
pre-building with versal harmonic map whose differential is
given by the spectral curve. 
\end{abstract}

\maketitle


\section{Introduction}

Let $X$ be a Riemann surface 
with a spectral covering $\Sigma \subset T^{\ast}X$ for the
group $SL(3)$. In \cite{KNPS2} we proposed in general terms
a reduction process that would construct a versal
$\Sigma$-harmonic map to an $SL(3)$ pre-building.
It was conjectured that if the $\Sigma$-spectral network \cite{GMN-Wall, GMN-SN, GMN-Snakes} has
no BPS states then the reduction process 
should be well-defined.

This conjecture would lead to a precise calculation of the
WKB exponents for singular perturbations whose spectral
curve has no BPS states, generalizing the known picture
\cite{Dumas} for
quadratic differentials and $SL(2)$. 

The purpose of the present paper is to provide more
details on the reduction process, particularly about
the combinatorial structure of the singularities that
occur and how they are arranged at each reduction step.  
We will show (although the proofs are sometimes only sketches) that the reduction steps are well-defined if
there are no BPS states. It is left for later to show that
the process finishes in finitely many steps. 

Assume $X$ is complete and simply connected. For the present
one should think of it as being the complex plane with 
$\Sigma$ a spectral covering similar to the one considered in our original example \cite{KNPS}. The case 
of the universal covering of a compact Riemann surface
would pose additional problems of non-finiteness of the
set of singularities, and the right notion of convergence seems
less clear.  

The process and results will be summed up in \S \ref{scholium}
and the reader is referred there. 

In the remainder of the introduction, we review 
the motivation for the reduction process considered here. 
A building $B$ for the group $SL(3)$ is a piecewise linear
cell complex that is covered by copies of the standard appartment $A=\rr ^2$, and indeed any two points 
of $B$ are contained in a common appartment. One of the main 
properties characterizing a building is that it is negatively curved. 

A {\em harmonic map} $h:X\rightarrow B$ is a continuous map such that
any point $x\in X$, except for a discrete set of singularities, admits a neighborhood $x\in U\subset X$ such that there is
an appartment $A\subset B$ with $h:U\rightarrow A\cong \rr ^2$ being a harmonic map.

The differential $dh$ is naturally a triple of real $1$-forms
$(\eta _1,\eta _2,\eta _3)$ with $\eta _1+\eta _2+\eta _3=0$. 
These are well-defined up to permutation. Now, these real harmonic forms are real parts of holomorphic $1$-forms $\eta _i=\Re \phi _i$ and the collection $\{ \phi _1,\phi _2,\phi _3\}$
defines the spectral curve $\Sigma \subset T^{\ast}X$. We say that $h$ is a  $\Sigma$-harmonic map. 

For a given spectral curve we would like to understand the 
$\Sigma$-harmonic maps to buildings. We conjectured in 
\cite{KNPS,KNPS2} that, under a certain genericity hypothesis, 
there should be an essentially uniquely defined 
map 
$$
h_{\phi}: X\rightarrow B_{\phi}^{\rm pre}
$$
depending only on the spectral curve $\Sigma = \{ \phi _1,\phi _2,\phi _3\}$ with the following properties: 
\newline
(1)\, the pre-building
$B^{\rm pre}_{\phi}$ is a negatively curved complex built out of enclosures in $A$ \cite{KNPS2}, and 
\newline
(2)\, 
any $\Sigma$-harmonic map to a building $h:X\rightarrow B$ 
factors through an embedding $B^{\rm pre}_{\phi}\rightarrow B$
isometric for the Finsler and vector distances. 

The conjectured 
genericity hypothesis for existence of $h_{\phi}$ 
is that the spectral network 
associated to $\Sigma$ should not have any BPS states
\cite{GMN-Wall, GMN-SN, GMN-Snakes}. 

Before getting to the proposed method for constructing 
$B^{\rm pre}_{\phi}$, let us consider the implications for exponents of WKB problems. There are several different ways of getting harmonic mappings to buildings, such as 
Gromov-Schoen's theory \cite{GromovSchoen}. Parreau 
interpreted boundary points of the character variety 
as actions on buildings \cite{Parreau}. In \cite{KNPS} we
extended Parreau's theory slightly for the situation of
WKB problems, getting a control on the differential. 
Suppose $\nabla _t$ is a singular perturbation
of flat connections, for $t$ a large parameter. 
There are two typical ways of getting $\nabla_t$,
the Riemann-Hilbert situation
$$
\nabla_t= \nabla _0 + t\varphi 
$$
or by solution of Hitchin's equations for the Higgs bundle
$(E,t\varphi )$. In either case, there is an associated 
limiting Higgs bundle $(E,\varphi )$ and we let $\Sigma$ be
its spectral curve. 

For $P,Q\in X$ let $T_{PQ}(t):E_P\rightarrow E_Q$ 
denote the transport for $\nabla _t$. For an ultrafilter 
$\omega$ on $t\rightarrow \infty$ define the
{\em exponent}
$$
\nu ^{\omega}_{PQ}:= \lim _{\omega} 
\frac{1}{t}\log \left\| T_{PQ}(t)\right\| .
$$
There is a similar {\em vector exponent} \cite{KNPS} that is 
a point in the positive Weyl chamber of $A$. 

The groupoid version of Parreau's theory \cite{Parreau, KNPS}
gives a map to a building
$h_{\omega}: X\rightarrow B_{\omega}$ such that
the Finsler distance (resp. vector distance) between 
$h_{\omega}(P),h_{\omega}(Q)$ is the exponent 
$\nu ^{\omega}_{PQ}$ (resp. the vector exponent).

We showed in \cite{KNPS} for the Riemann-Hilbert
situation 
that $h_{\omega}$ is a $\Sigma$-harmonic map. 
Mochizuki \cite{Mochizuki} 
showed this for the Hitchin WKB problem. 

If there exists a $\Sigma$-harmonic map $h_{\phi}$
satisfying the properties (1), (2) above, then it follows that
$\nu^{\omega}_{PQ}$ is calculated as the Finsler distance
between $h_{\phi}(P),h_{\phi}(Q)$. In particular, it depends
only on the spectral curve $\Sigma$. Indepenence of the choice
of ultrafilter means that the ultrafilter limit used to
define the exponent is actually a limit, and we obtain 
a calculation of the WKB exponents for our singular perturbation. 

Turn now to the reduction process for constructing $B^{\rm pre}_{\phi}$. A conjectural yet detailed picture of this construction was set out in \cite{KNPS2} and readers are referred there for a full explanation. The first step was
to make an initial construction. That is essentially what
we shall be calling $Z^{\rm init}$ below, although the
initial construction for the pre-building has to include
additional small parallelogram-shaped regions corresponding to
the folded pieces $\widetilde{Q}_i$ that we'll meet in
\S \ref{sec-initial} below. Our $Z^{\rm init}$ has
these trimmed off. 

The main problem of the initial construction is that it contains points of positive curvature, referred to as $4$-fold points or
$\four _2$ points below. These have to collapse in some way
under any harmonic map to a building $B$ since $B$ is negatively curved. The construction of the pre-building consists of successively doing such a collapsing operation. 

Unfortunately, the direction of collapsing at a $4$-fold point is not well-defined, rather there are two possibilities. For that reason, the notion of {\em scaffolding} was introduced in \cite{KNPS2}. The initial scaffolding formalizes the existence of small neighborhoods $U$ that  map, in an unfolded way, into
appartments of the building. It follows that $U$ should not be folded in the map $h_{\phi}$ to $B^{\rm pre}_{\phi}$.
The edges gotten by gluing together sectors are, on the other hand, to be folded further, and this collection of data is complete: every edge is either marked as folded or unfolded (open) in the scaffolding. 

The scaffolding tells us which direction to fold at a $\four_2$ point. The remaining difficulty is to propagate the information of the scaffolding into the new constructions obtained by collapsing. It was conjectured in \cite{KNPS2} that this should be possible, under the hypothesis of absence of BPS states in the original spectral network. 

The purpose of the present paper is to prove this conjecture on propagation of the scaffolding. We show how to make a series
of reduction steps and how to propagate the scaffolding and
other required information so that the series of reduction steps is well-defined.

The present work does not yet result in a full construction of $B^{\rm pre}_{\phi}$. Notably missing is a convergence statement saying that the process stops in a  (locally) finite number 
of steps. Some parts of our arguments are also sketches rather than full proofs, and we don't provide here 
a justification for the choice of initial construction (Principle \ref{chooseinitial}). 

It was observed in \cite{KNPS2} that each step of the reduction process towards the pre-building, could be accompanied by a trimming of the just glued-together parallelograms. If one does that, then the sequence of constructions is a sequence of $2$-dimensional surfaces. This point of view will be useful for the
program of generalizing Bridgeland-Smith's work on stability conditions, discussed briefly in \S \ref{sec-further}
at the end. Also, the combinatorics of the reduction steps are all contained in this sequence of surfaces. Therefore, in the
present paper we shall include the trimming operation as part of our reduction steps. In order to construct the pre-building $B^{\rm pre}_{\phi}$ one should put back in the pieces that were trimmed off---the procedure for doing that was explained in \cite{KNPS2}.

\section{Spectral networks in $X$}

Start by considering spectral networks in $X$. 
Away from the branch points, the spectral curve $\Sigma$ 
consists of three holomorphic $1$-forms $\phi _1,\phi _2,\phi _3$ and setting $\phi  _{ij}:= \phi _i-\phi _j$ 
these define {\em foliation lines} $f_{ij}$ where 
$\Re \phi _{ij}=0$. Assume the ramification of $\Sigma$ consists of simple branch points.
At a branch point, two indices are interchanged, picking out
one of the foliations $\Re \phi _{ij}=0$ whose singular leaves starting from the branch point are the initial edges of the
spectral network. 

Away from the {\em caustics} where the three foliation lines
are tangent, the coordinates $\Re \phi _i$ define a
flat structure on $X$ by local identification with the
standard apartment 
$$
A:= \{ (x_1,x_2,x_3)\in \rr ^3, \;\;  
x_1+x_2+x_3=0\} \cong \rr ^2
$$
for buildings of the group $SL(3)$. We use this flat structure when speaking of angles.

We are going to add some extra singularities,
so suppose given a finite subset $S\subset X$ 
and for each $p\in S$ one of the three foliations $f_p$. 
A {\em spectral network graph} is a map from a trivalent
graph $G$, that can have endpoints as well as at most one infinite end,
to $X$ that sends endpoints of $G$ to elements of $S$,
that sends edges to foliation lines, and that sends trivalent
vertices to {\em collisions} \cite{GMN-SN}, points where
foliation lines $f_{ij},f_{jk},f_{ik}$ meet at $120^{\circ}$
angles. 
We require that an endpoint
of the graph going to $p\in S$ has adjoining edge going to a foliation line for the given foliation $f_p$.

A foliation line in $X$ is an {\em SN-line} if there exists
a spectral network graph such
that the given foliation line is in the image of
an edge of the graph that is either adjacent to an 
endpoint, or is an  infinite end.

A {\em BPS state} \cite{GMN-Wall, GMN-SN, GMN-Snakes} 
is a compact 
spectral network graph. Our main hypothesis will be that
these don't exist. 

We now note how to add singularities while conserving this
hypothesis.

The original set $S_0$ of singular points is equal
to the set of ramification points of the spectral curve,
assumed to be simple ramifications. The foliation line $f_p$ at a ramification point is the one determined by the two 
sheets of the spectral curve that come together at that point. 

The following proposition will allow us successively
to add points to the set of singularities. 

\begin{proposition}
\label{addsn}
Suppose given a set of singularities $S_{i-1}$ such that the
resulting spectral network doesn't have BPS states. 
Choose a point $p_{i}$ in general position along the
interior of a 
caustic curve, let $f_{p_{i}}$ be one of the 
foliation lines at $p_{i}$ and let $S_{i}:= 
S_{i-1}\cup \{ p_{i}\}$. Then the spectral network
associated to $S_{i}$ also doesn't have any BPS states.
\end{proposition}
\begin{proof}
Suppose we are given a spectral network graph 
$\beta : G\rightarrow X$
that is a BPS state for $S_i$. Consider a nearby point
$p_i(\epsilon )$ obtained  by moving it a small distance
$\epsilon$ in one
direction along the caustic curve (the caustic is generically
transverse to the foliation lines otherwise it would constitute a BPS state itself). 
We may assume that 
the BPS state follows to a nearby one $\beta (\epsilon ):G\rightarrow X$. The foliation lines are defined by differential
forms $\Re \phi _{ij}$ so this gives us a way of measuring transverse distances; in these terms it makes sense to
talk about the distance between one of the foliation lines and
an adjacent one for the same foliation. We should choose
for each edge of the graph a sign for the form in question. 
Now, label edges of the graph $e\in {\rm edge}(G)$ by integers $k(e)$ such that 
the edge $e$ moves by $k(e)\epsilon$. An edge $e'$ adjoining
an endpoint of the graph
that goes to a previous singularity $q\in S_{i-1}$ has to be
labeled with $k(e')=0$, since $q$ doesn't move with $\epsilon$.
There is a balancing condition on the labels at the
collision points. The edges that end in $p_i$ from one direction
have to be labeled by $k=1$ whereas the edges that end in $p_i$ from the other direction have to be labeled by $k=-1$.
However, the balancing condition at the trivalent vertices
plus the condition that all the other endpoint labels are zero,
means that the number of $k=1$ labels and the number of $k=-1$
labels at $p_i$ have to be equal. Therefore, we can pair up edges coming in from one direction with edges going out in the other, glueing these edges together pairwise. This
results in a new graph $G'\rightarrow X$ that is a BPS
state for the previous set of singularities $S_{i-1}$. 
\end{proof}

\noindent
{\bf Boundedness}---Let us mention a boundedness hypothesis
that will be useful. Suppose that there exists a compact
subset $K\subset X$ whose boundary $\partial K$ consists
of foliation lines, such that the corners of $\partial K$
are convex in the sense that each corner consists of 
foliation lines separated by two sectors inside $K$. And, we suppose that all ramification points of the spectral curve, and all caustics, are contained in the interior of $K$. Since the additional singularities were chosen on caustics, it follows that $S\subset K$. 

With this hypothesis, a spectral network graph has no collisions outside of $K$, and any edge that leaves $K$ continues as an infinite end. Indeed, all caustics are contained in
the interior of $K$, so $X-K$ has a flat structure modeled on 
the standard appartment $A$. Suppose $x\in X-K$ were a collision
point, say the first one outside of $K$.  Then the two incoming edges must exit from $K$, but two distinct 
foliation lines that exit from $K$ and meet, must exit from the
same edge because of the convexity hypothesis on the corners. 
If they exit from the same edge and meet, then they must meet at a $60^{\circ}$ angle, contradicting the hypothesis that they form incoming edges of a collision. Thus, such a collision cannot occur. 

Referring to the reduction steps that will be discussed later, any regions that are to be folded together have to stay inside $K$ by the same kind of considerations. Therefore, we may
follow our compact subset along into the sequence of constructions $Z$ that will occur; the compact subset will contain all modifications, and $X-K$ remains untouched as a flat space contained in each of our series of constructions. 

Although we don't discuss the question of convergence in the present paper, the existence of a series of compact subsets
that contain all the modifications means that one could envision an argument using decrease of the area to conclude termination of the reduction process. That would require bounding from below the size of the reduction steps.

\section{The initial construction}
\label{sec-initial}

The {\em initial construction} is a
construction $Z^{\rm init}$ obtained by a first process
of folding together certain regions in $X$ and trimming
away the resulting pieces. We consider a collection of
regions $Q_i\subset X$ covering the caustics, such that
$Q_i$ are bounded by foliation lines meeting in
two points $q_i,q'_i$ symmetrical with respect to the caustic.
In any $\phi$-harmonic map to a building
$h:X\rightarrow B$, $q_i$ and $q'_i$ have to map
to the same point, and indeed $h|_{Q_i}$ has to factor
through a map $\widetilde{Q}_i\rightarrow B$ where
$\widetilde{Q}_i$ is $Q_i$ folded in two 
along the caustic \cite{KNPS,KNPS2}.

Here is a picture of the regions $Q_i$  in $X$

\hspace*{1cm}
\includegraphics{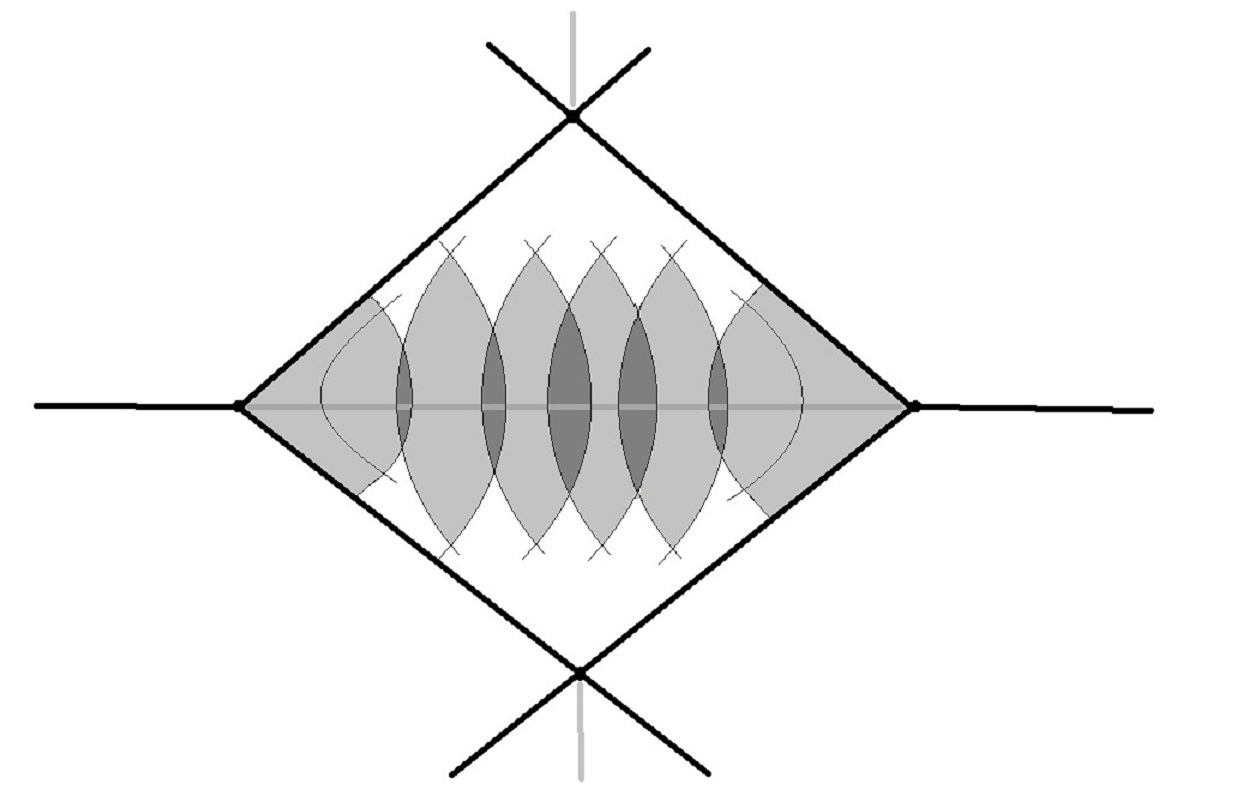}

and here is what their images look like in any harmonic
map to a building: 

\hspace*{1cm}
\includegraphics{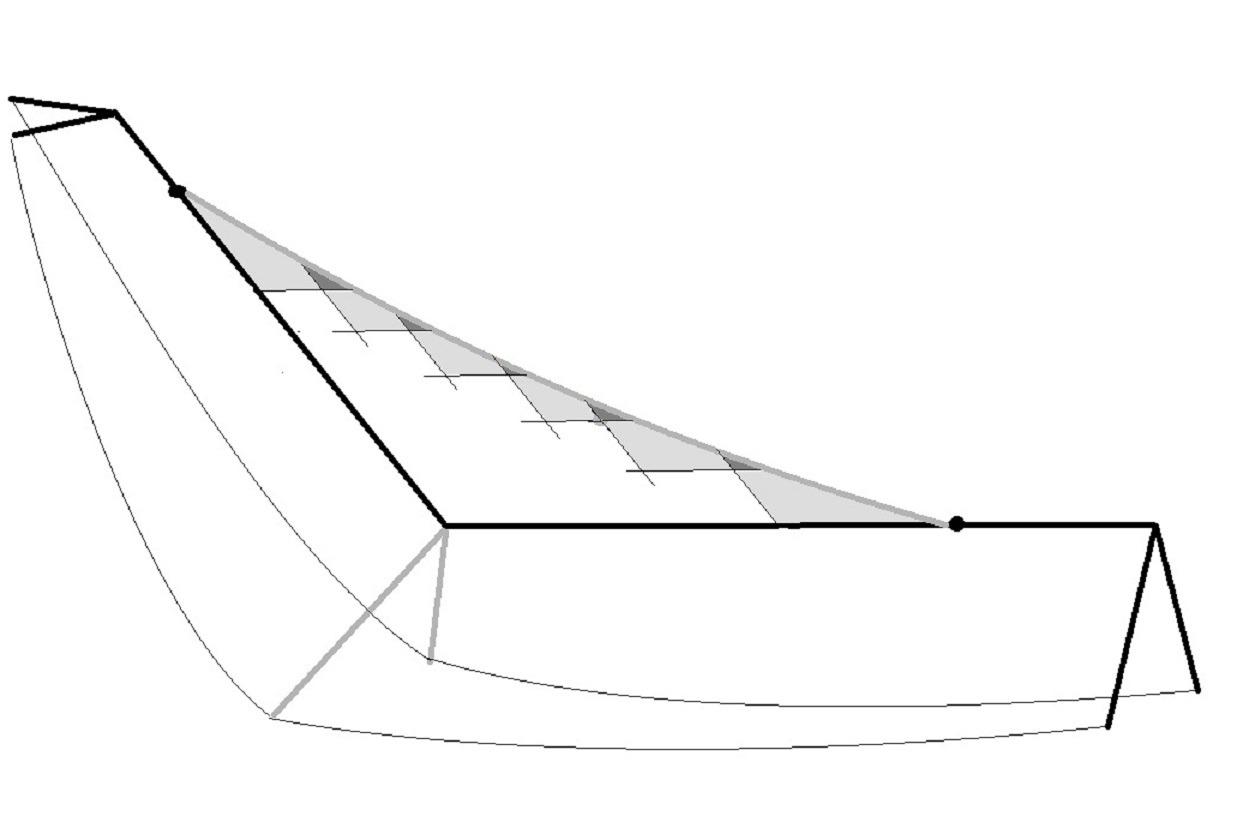}

We assume that the boundaries of the regions $Q_i$ are
formed by new spectral network lines gotten after adding
singularities along the caustics according to Proposition
\ref{addsn}. That will be used for the refracting property in
\S \ref{sec-refracting}.

Let $\widetilde{X}$ denote the quotient of $X$ by
the equivalence relation induced by the quotients
$Q_i\rightarrow \widetilde{Q}_i$. Now let $X^{\rm init}$
be the result of trimming off the folded-together
pieces, in other words it is the closure in $\widetilde{X}$
of $X-\bigcup Q_i$, or equivalently the image
of $X-\bigcup Q^{o}_i$. 

The space $X^{\rm  init}$ is a topological surface, 
and we have cut out the caustics. Thus, $X^{\rm init}$ is
provided with a geometric structure locally modeled on
the standard appartment $A$; the three foliation lines at any point correspond to the three standard directions in $A$. 
In particular $X^{\rm init}$ has a flat metric. The conformal
structure for this flat metric is different from the original
structure of Riemann surface on $X$. It has singular points of
positive and negative curvature, namely $8$-fold points
of negative curvature 
whose total angle is $480^{\circ}$ and $4$-fold points 
of positive curvature whose
total angle is $240^{\circ}$. 

View $X^{\rm init}$ as corresponding to a construction, i.e. a 
presheaf on the site of enclosures as discussed in
\cite{KNPS2}. 
More precisely we have a construction $Z^{\rm init}$ whose
usual set of points is $Z^{\rm init}(p)=X^{\rm init}$. 
It is provided with a map 
$$
h^{\rm init}:X-\bigcup Q^o_i \rightarrow Z^{\rm init}(p)
$$
that
sends foliation lines in $X$ to segments in $Z^{\rm init}$. 

This construction will be the starting point of our reduction process. It is provided with a {\em scaffolding} 
\cite{KNPS2}: certain
edges (the image of $\partial \bigcup Q^o_i$) 
are designated as ``fold edges'' while all the other
edge germs are designated as ``open'' or ``unfolded''. 
If $h:X\rightarrow B$ is any $\phi$-harmonic map to 
a building, then its restriction to $X-\bigcup Q^o_i$
factors through a map $Z^{\rm init}\rightarrow B$ that
respects the scaffolding, in the sense that fold edges are
folded and open edges are unfolded. 

\section{Structures}

Let $Z$ be a construction. We say that $Z$ is
{\em complete} if the associated metric space is complete.
It is {\em normal} if the link at any point is connected.  
We say that $Z$ is {\em ecarinate}
if, at any edge there are two half-planes.  
Notice that $Z^{\rm init}$ satisfies these conditions, and
our reduction process will conserve them, so
let us consider only complete normal ecarinate and simply connected constructions
$Z$.  In particular the set of usual points
$Z(p)$ is a complete $2$-manifold.

If $z\in Z(p)$ is a point then the link
$Z_z$ is a connected 
graph such that each vertex (corresponding to a germ of edge at $z$) is contained in two edges (corresponding to germs of sectors at $z$). Therefore, the link is a polygon.
By the parity property the polygon has an even number of edges. 
We assume that the number of edges  in the link at any point is $4$, $6$ or $8$. Most points are {\em flat}, meaning that their links are hexagons. The other points are called $4$-fold, or $8$-fold respectively. The $4$-fold points are points of positive curvature, and the $8$-fold points are points of negative curvature. 

A {\em scaffolding} consists of the following
structures:
\newline
(1)\;  A marking of each edge in $Z$ as either {\em open} (o) or {\em folded} (f), such that the fold edges are those
from a discrete
collection ${\bf F}$ of straight edges in $Z$. 
\newline
(2)\; An orientation assigned to each fold edge.
\newline
(3)\; A marking of some subset of the fold edges said to be {\em refracting}.

Our initial construction $Z^{\rm init}$ contains
a scaffolding in which all the edges are already
marked as refracting, see \S \ref{sec-refracting} below, 
and our reduction process will preserve
the refracting condition so we henceforth consider
only fully refracting scaffoldings. 

We assume that $Z^{\rm init}$ and its scaffolding
have the property that there exists a
harmonic mapping to a building such that the fold edges are
folded and the open edges are unfolded. This constrains the
local type of singularities. However, the number of possibilities is still rather large. 

We describe
here a list standard examples that will be sufficient for
our reduction process---the statement that we remain within this standard list is indeed one of the main conclusions of our 
treatment in the present paper. 

The notation will consist of a boldface number saying how many sectors there are, and a subscript saying how many folded lines in the scaffolding there are. Arrows in the pictures indicate
the orientations of the scaffolding edges. 

For example, 
points  of type $\six _0$ and $\eight_0$ are respectively 
$6$-fold and $8$-fold points with no adjoining fold edges.
A point of type $\six_2$ is a $6$-fold point with a single
straight fold edge (it is nonsingular); 
a $\six_3$ point has three folded edges
separated by $120^{\circ}$, and a $\six_4$ point is the same with an additional folded edge. These may be pictured as follows:
$$
\setlength{\unitlength}{.3mm}
\begin{picture}(100,100)

\qbezier(50,50)(50,50)(68,80)
\qbezier(50,50)(50,50)(68,20)
\qbezier(50,50)(50,50)(32,80)
\qbezier(50,50)(50,50)(32,20)

\linethickness{.8mm}
\qbezier(50,50)(50,50)(90,50)

\qbezier(50,50)(50,50)(10,50)
\thinlines

\qbezier(25,50)(25,50)(29,56)
\qbezier(25,50)(25,50)(29,44)

\qbezier(75,50)(75,50)(79,56)
\qbezier(75,50)(75,50)(79,44)

\put(5,15){$\six_2$}

\end{picture}
\hspace*{.8cm}
\setlength{\unitlength}{.3mm}
\begin{picture}(100,100)

\put(50,50){\circle*{5}}

\qbezier(50,50)(50,50)(68,80)
\qbezier(50,50)(50,50)(68,20)
\qbezier(50,50)(50,50)(10,50)

\linethickness{.6mm}

\qbezier(50,50)(50,50)(32,80)
\qbezier(50,50)(50,50)(32,20)

\linethickness{.8mm}
\qbezier(50,50)(50,50)(90,50)
\thinlines

\qbezier(41,65)(41,65)(46,73)
\qbezier(41,65)(41,65)(33,65)

\qbezier(41,35)(41,35)(49,36)
\qbezier(41,35)(41,35)(36,42)

\qbezier(75,50)(75,50)(79,56)
\qbezier(75,50)(75,50)(79,44)

\put(5,15){$\six_3$}

\end{picture}
\hspace*{.8cm}
\setlength{\unitlength}{.3mm}
\begin{picture}(100,100)

\put(50,50){\circle*{5}}

\qbezier(50,50)(50,50)(68,80)
\qbezier(50,50)(50,50)(68,20)

\linethickness{.6mm}

\qbezier(50,50)(50,50)(32,80)
\qbezier(50,50)(50,50)(32,20)

\linethickness{.8mm}
\qbezier(50,50)(50,50)(90,50)
\qbezier(50,50)(50,50)(10,50)
\thinlines

\qbezier(41,65)(41,65)(46,73)
\qbezier(41,65)(41,65)(33,65)

\qbezier(41,35)(41,35)(49,36)
\qbezier(41,35)(41,35)(36,42)

\qbezier(75,50)(75,50)(79,56)
\qbezier(75,50)(75,50)(79,44)

\qbezier(25,50)(25,50)(29,56)
\qbezier(25,50)(25,50)(29,44)

\put(5,15){$\six_4$}

\end{picture}
$$

Next we picture the $8$-fold and $4$-fold points. Note
that the pictures cannot be conformally correct for the 
angles; all drawn sectors represent sectors of $60^{\circ}$
in $Z$. 

The $8$-fold points have at most $2$ fold edges; and if
there are $2$ of them, they are separated by either $1$ or
$4$ sectors. The fold edge orientations go outward,
so after the $\eight_0$ picture there are three possibilities:
$$
\setlength{\unitlength}{.3mm}
\begin{picture}(100,100)

\put(50,50){\circle*{5}}

\qbezier(50,50)(50,50)(50,90)
\qbezier(50,50)(50,50)(50,10)

\qbezier(50,50)(50,50)(10,50)

\qbezier(50,50)(50,50)(80,80)
\qbezier(50,50)(50,50)(80,20)
\qbezier(50,50)(50,50)(20,80)
\qbezier(50,50)(50,50)(22,22)

\linethickness{.8mm}
\qbezier(50,50)(50,50)(90,50)
\thinlines

\qbezier(75,50)(75,50)(71,56)
\qbezier(75,50)(75,50)(71,44)

\put(5,15){$\eight_1$}

\end{picture}
\hspace*{.8cm}
\setlength{\unitlength}{.3mm}
\begin{picture}(100,100)

\put(50,50){\circle*{5}}

\qbezier(50,50)(50,50)(50,90)
\qbezier(50,50)(50,50)(50,10)

\qbezier(50,50)(50,50)(10,50)

\qbezier(50,50)(50,50)(80,80)
\qbezier(50,50)(50,50)(80,20)
\qbezier(50,50)(50,50)(20,80)
\qbezier(50,50)(50,50)(22,22)

\linethickness{.8mm}
\qbezier(50,50)(50,50)(90,50)
\qbezier(50,50)(50,50)(10,50)
\thinlines

\qbezier(75,50)(75,50)(71,56)
\qbezier(75,50)(75,50)(71,44)

\qbezier(25,50)(25,50)(29,56)
\qbezier(25,50)(25,50)(29,44)

\put(5,15){$\eight_2$}

\end{picture}
\hspace*{.8cm}
\setlength{\unitlength}{.3mm}
\begin{picture}(100,100)

\put(50,50){\circle*{5}}

\qbezier(50,50)(50,50)(50,90)
\qbezier(50,50)(50,50)(50,10)

\qbezier(50,50)(50,50)(10,50)

\qbezier(50,50)(50,50)(80,80)
\qbezier(50,50)(50,50)(80,20)
\qbezier(50,50)(50,50)(20,80)
\qbezier(50,50)(50,50)(22,22)

\linethickness{.8mm}
\qbezier(50,50)(50,50)(90,50)

\linethickness{.6mm}
\qbezier(50,50)(50,50)(80,80)
\thinlines

\qbezier(75,50)(75,50)(71,56)
\qbezier(75,50)(75,50)(71,44)

\qbezier(71,71)(71,71)(70,61)
\qbezier(71,71)(71,71)(61,70)

\put(5,15){$\eight'_2$}

\end{picture}
$$

Recall \cite{KNPS2} that
at a $4$-fold point, at least two of the four edges are folded,
and if an edge is folded then so is the opposite one. 
Our standard type is when only two edges are folded, 
and the edge orientations are inwards towards the
singularity, so it has the following picture:
$$
\setlength{\unitlength}{.3mm}
\begin{picture}(100,100)

\put(50,50){\circle*{5}}

\qbezier(50,50)(50,50)(50,87)
\qbezier(50,50)(50,50)(50,13)

\linethickness{.8mm}
\qbezier(50,50)(50,50)(87,50)
\qbezier(50,50)(50,50)(13,50)
\thinlines

\qbezier(30,50)(30,50)(26,56)
\qbezier(30,50)(30,50)(26,44)

\qbezier(70,50)(70,50)(74,56)
\qbezier(70,50)(70,50)(74,44)

\put(5,15){$\four_2$}

\end{picture}
$$

\begin{definition}
\label{singdef}
We say that the singularities of the scaffolding
are
{\em initial} if at any point of $Z$
the picture is either $\six_0$ (a smooth point
not on the scaffolding), $\six_2$ (an interior point of an edge of the scaffolding), 
an $8$-fold point of the form 
$\eight_1,\eight_2,\eight'_2$, or a $4$-fold point 
$\four_2$. We say that the singularities of the scaffolding are 
{\em standard}
if at any point the local picture is one of the
ones we have described, namely:
$$
\six_0,\; \six_2,\; \six_3,\; \six_4,\; 
\eight_0, \; \eight_1,\; \eight _2,\; \eight'_2,  
\mbox{ or }\four_2 
$$
shown above. 
\end{definition}

Included in the above definitions are compatibility of the
orientations of the fold lines with the singularity
types as drawn in the pictures. Recall that
edges are oriented outward at $8$-fold points and inward at $4$-fold points. For initial scaffoldings we
obtain the following characterization: 

\begin{remark}
\label{postcaustic}
Given a scaffolding with initial singularities,
the collection of fold lines decomposes into a disjoint
union of piecewise linear curves called the {\em post-caustics},
whose endpoints are of type $\eight_1$, and along which the
$4$-fold points alternate with $8$-fold points. In particular, the number of $8$-fold points on a connected post-caustic
is $1$ more than the number of $4$-fold points. 
\end{remark}

\section{The refraction property}
\label{sec-refracting}

Using Proposition \ref{addsn} to add spectral network lines
along the boundaries of the regions $Q_i$ used to define
the initial construction, will give that the
scaffolding of the initial construction $Z^{\rm init}$
is a fully refracting scaffolding. The ``refracting properties'' with respect
to the spectral network are:
\newline
(R1)\;\; At singular points of the scaffolding, all unfolded
edges are initial edges of the spectral network; and
\newline
(R2)\;\; When a spectral network line crosses over a
fold edge, it can continue on the other side in either one of
the two available directions. 

Suppose we are given a construction $Z$ and
a fully refracting scaffolding. A {\em spectral network graph for the refracting
spectral network} is a map from a graph $G\rightarrow Z$
satisfying the following properties:
\newline
---edges of $G$ go to segments in $Z$ that
are straight except when they cross the fold edges; 
\newline
---edges do not go along fold edges of the scaffolding; 
\newline
---when edges 
cross over the fold edges they can ``refract'', that is to
say they go out
of the fold edge on the other side in either one of the two directions;
\newline
---trivalent vertices of the graph go to collisions
at $\six_0$ points i.e. points 
not on fold edges of the scaffolding; and
\newline
---endpoints of the
graph go to singular points of the scaffolding or
$\eight_0$ points of $Z$, with the adjoining edges going outward in any non-fold directions. 

Before choosing the regions $Q_i\subset X$ that will be
folded and trimmed to get $X^{\rm init}$, add points 
to the set of singularities of our spectral
network, using Proposition \ref{addsn}. Add these points along
the caustic curves at the places where the boundaries
of the $Q_i$ meet the caustics, so that the
boundary curves of the $Q_i$ become spectral network curves.
Notice furthermore that at the corners $q_i,q'_i$,
all three directions outward to the rest of $X$ are spectral
network lines, two from continuing the boundary curves and
the third by collision of the two boundary curves. 

Our above somewhat heuristic discussion leads to the following
principle. 

\begin{principle}
\label{nobps}
If the original spectral network of $X$ had no BPS states, then
the refracting spectral network of $Z^{\rm init}$ has no BPS states. 
\end{principle}

We also need to know something about the arrangement of
singularities of the scaffolding for $Z^{\rm init}$. 
We state this as another principle:

\begin{principle}
\label{chooseinitial}
The initial construction $Z^{\rm init}$ may be chosen to
have only initial singularities (Definition \ref{singdef}).
The fold edges
of the scaffolding are thus arranged into  
post-caustics as described in Remark \ref{postcaustic}. 
\end{principle}

We don't give here a formal justification for the possibility
of choosing the initial construction in this way, but note that
it is what happens in the pictures we have considered.
For example, 
see \cite{KNPS2} for a picture leading to an $\eight'_2$ point.

\section{Reduction}

This section begins the discussion of a step in the reduction
process. The first question is to show the existence
of a $\four_2$ point about which some collapsing can be done.

\begin{proposition}
\label{prop-noloops}
Suppose given a construction with a refracting scaffolding,
such that there are no BPS states in the resulting spectral
network. Suppose that the fold edges of the scaffolding are oriented and the singularities are all from our standard  list. 
Make a directed graph using the singularities as vertices,
except that we separate a $\six_4$ singularity into two
vertices. The edges of the graph are the fold edges with
their orientations; at a $\six_4$ point the
``spine'' (consisting of the two fold edges that are opposite) 
goes to one of the vertices and the other two edges
go to the other one. This directed graph has 
no directed loops. 
\end{proposition}
\begin{proof} {\em (Sketch)}---Consider a path parallel to
a directed loop, just to one side of it. This will satisfy
the collision and refracting conditions, so taken together
with the appropriate initial SN lines coming from
singularities, it constitutes a BPS state. 
\end{proof}

\begin{corollary}
\label{existsfourtwo}
If there are no BPS states and if the set of fold lines in the 
scaffolding is nonempty, then there
must be a $\four_2$ point.
\end{corollary}
\begin{proof}
In the directed graph described in 
Proposition \ref{prop-noloops}, our hypothesis that there
are no BPS states implies that there are 
no directed loops. Therefore, the orientations
of edges make the graph into a poset; since it is finite it has a minimal vertex. The only type of point on the scaffolding that has all fold edges pointed inward is a $\four_2$ point,
therefore there exists a $\four_2$ point.
\end{proof}

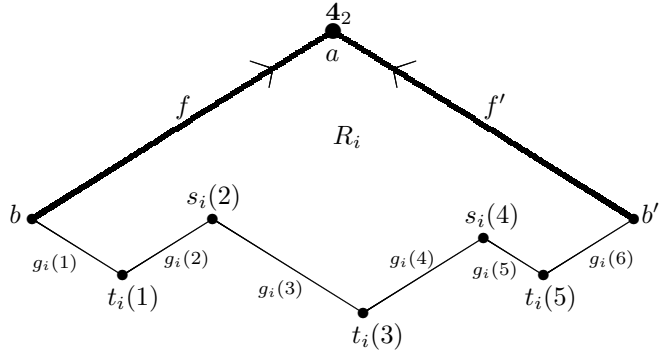
\begin{figure}
$$
\setlength{\unitlength}{.5mm}
\begin{picture}(200,110)

\put(20,50){\circle*{3}}
\put(180,50){\circle*{3}}
\put(100,100){\circle*{4}}

\linethickness{.6mm}
\qbezier(20,50)(20,50)(100,100)
\qbezier(180,50)(180,50)(100,100)
\thinlines

\qbezier(84,90)(84,90)(78,92)
\qbezier(84,90)(84,90)(83,85)

\qbezier(116,90)(116,90)(122,92)
\qbezier(116,90)(116,90)(117,85)

\put(44,35){\circle*{3}}
\qbezier(20,50)(20,50)(44,35)

\put(68,50){\circle*{3}}
\qbezier(44,35)(44,35)(68,50)

\put(108,25){\circle*{3}}
\qbezier(68,50)(68,50)(108,25)

\put(140,45){\circle*{3}}
\qbezier(108,25)(108,25)(140,45)

\put(156,35){\circle*{3}}
\qbezier(140,45)(140,45)(156,35)

\qbezier(156,35)(156,35)(180,50)

\put(98,103){$\four_2$}
\put(98,92){$a$}
\put(58,78){$f$}
\put(140,78){$f'$}

\put(14,49){$b$}
\put(182,49){$b'$}

\put(61,54){$s_i(2)$}
\put(135,49){$s_i(4)$}

\put(40,27){$t_i(1)$}
\put(105,17){$t_i(3)$}
\put(151,27){$t_i(5)$}

\put(20,37){$\scriptstyle g_i(1)$}
\put(55,38){$\scriptstyle g_i(2)$}
\put(80,30){$\scriptstyle g_i(3)$}
\put(115,38){$\scriptstyle g_i(4)$}
\put(137,35){$\scriptstyle g_i(5)$}
\put(168,38){$\scriptstyle g_i(6)$}

\put(100,70){$R_i$}

\end{picture}
$$
\vspace*{-1.4cm}
\caption{\label{mcr} Maximal collapsing region}
\end{figure}

We now consider an extended 
collapsing operation at a $\four _2$ point
$a$. 
For this, consider two pieces $R_1$ and $R_2$
that match up, and share a
common vertex $a$ and common edges $f$ and $f'$. Let $b$ (resp. $b'$) denote the endpoint of $f$ (resp. $f'$) different from 
$a$. 
The $R_i$ are assumed to be constructions that can be
considered as isomorphic to 
a subset $R$ in an abstract parallelogram $P$. 
Let us also label the corresponding vertex $a$ and the
corresponding edges $f,f'$ in $P$. The embeddings $R_i\cong R \subset P$ preserve $a,f,f'$. We can now describe the configuration of $R$: the vertex $a$ is an obtuse vertex of $P$, and $f,f'$ are the full edges of $P$ meeting $a$. 
Hence the points also labelled $b,b'\in P$ are the two
acute vertices. We assume that $R$ is a union of finitely many
sub-parallelograms $R(2j-1)\subset P$ (see the picture below for the numbering) such that $R(2j-1)$ contains $a$
as obtuse vertex, that is to say the edges of $R(2j-1)$ will be 
segments in $f,f'$ starting at $a$. Let $t(2j-1)$ denote the
vertices of $R(2j-1)$ opposite to $a$. We assume that the
order $R(1),R(3),\ldots , R(2k+1)$ is such that $t(1)$ is
on the same edge of $P$ as $b$, and they go in a consecutive sequence
until $t(2k+1)$ is on the same edge of $P$ as $b'$. 
The $t(1),t(3),\ldots , t(2k+1)$ are the other convex corners
of $R$ after $a,b,b'$. Let $s(2),s(4),\ldots , s(2k)$ denote the
concave corners of $R$ in between them, so that $s(2j)$ lies
between $t(2j-1)$ and $t(2j+1)$. 
Let $g(1)$ be the edge from $b$ to $t(1)$ and let 
$g(2k+2)$ be the edge from $t(2k+1)$ to $b'$. Let 
$g(2j)$ be the edge from $t(2j-1)$ to $s(2j)$ and
$g(2j+1)$ be the edge from $s(2j)$ to $t(2j+1)$. Thus,
$g(2j)$ and $f$ are parallel, and $g(2j-1)$ and $f'$ are parallel. 

Now consider the same points in the regions $R_i$, indicated
as $s_i(2j)$ and $t_i(2j-1)$, with edges $g_i(j)$.
The configuration of this maximal collapsing pair of regions 
may be pictured as in Figure \ref{mcr}.

We assume that $R$ is a maximal such region with corresponding
regions $R_i\subset Z$, such that the following conditions are satisfied:
\newline
---there are no singularities in the interior of $R_i$;
\newline
---the only singularities on the interior of the edges 
$f,f'$ are
$\six_4$ points where $f$ or $f'$ is the straight fold edge
(spine). 

\begin{proposition}
\label{properties}
Under the above maximality conditions, we have the following properties:
\newline
1. The vertices $b,b'$ are singularities, and furthermore
these singularities are not $\six_4$ points with edge 
$f$ or $f'$ on the spine;
\newline
2. For each $j$ there is exactly one of the $s_1(2j),s_2(2j)$ that
is an $8$-fold singularity, and the other is nonsingular;
\newline
3. There are no $8$-fold singularities in the interiors
of the edges $g_i(j)$;
\newline
4. If $q_1$ 
is a $4$-fold or singular 
$6$-fold point on an edge $g_1(j)$ of $R_1$ 
then the corresponding point $q_2$ on $R_2$ is not a singularity,
and vice-versa, and this also applies for corners $t_i(2j-1)$. 
\end{proposition}
\begin{proof}
Notice in general that there can't be a fold edge going from 
a point on one of the edges $g_i(j)$ into $R_i$
in the middle direction between the directions of $f$ and $f'$
(vertical in Figure \ref{mcr});
that would have to meet the $4$-fold point making it of type
$\four _4$, or meet an edge in a $\six_4$ point but 
oriented in the wrong direction. 

It follows that two corresponding points $q_1\in R_1$ and $q_2\in R_2$ can't both be singularities, for they are joined
by a common foliation line that refracts at an edge $f,f'$;
we have seen that it can't be a fold line so it would be an initial SN-line from both singularities resulting in a BPS state. We get 4 and part of 2.

For 1, if $b$ (resp. $b'$) were nonsingular, or a 
$\six_4$ point
with edge $f$ (resp. $f'$) along 
the spine, then one could continue the regions 
$R_i$ further along the corresponding edge. 

For 2, if both points
are nonsingular then the regions are not maximal, so one must be singular, say on $R_1$. 
However, if it were a $6$-fold point then both lines
going into $R_1$ parallel to $f$ and $f'$ would be fold lines;
these would have to meet $f$ and $f'$ in $\six_3$ or $\six_4$ points, resulting in fold lines going back in the opposite direction in $R_2$. These two would meet in the corresponding point of $R_2$ but that would have to be singular, contradicting the statement above. 

For 3, notice that if some segments of one of these edges are folded, then the orientations of the fold segments are all the same (it follows from the consideration of the first paragraph).
But if we had an $8$-fold point on the interior of one of these edges, it would send out either a fold line or an SN-line, and
that would contradict the existence of the $8$-fold singularity given in part 2. 
\end{proof}

We now consider how fold lines, SN lines and singularities
can be arranged along a pair of edges $g_1(j),
g_2(j)$. Each edge $g_i(j)$ 
has two endpoints,
$q_i^{(0)}$ that is either 
$s_i(j'')$
or $b,b'$, and $q_i^{(m_{ij}+1)}=t_i(j')$.

\begin{lemma}
\label{onesn}
One of the two edges $g_1(j)$ or $g_2(j)$ is an SN line, 
and has
no singularities in its interior (so $m_{ij}=0$),
or at the endpoint $q_i^{(1)}$. 
It points towards the endpoint $q_i^{(1)}=t_i(j')$. 
\end{lemma}
\begin{proof}
One of the two $q_i^{(0)}$ 
is an $8$-fold point. At our admissible
possibilities $\eight _0,\eight_1,\eight_2,\eight'_2$ there are never two fold edges separated by three sectors. Therefore, at this point either the outgoing edge along $g_i(j)$, 
or the edge that goes into $R_i$ in the opposite direction 
(that is to say separated by three sectors interior 
to $R_i$), are SN-lines. If $g_i(j)$ 
is an SN-line then we obtain the desired conclusion. Notably, there are no singular points along this edge otherwise that would create a BPS state. In the other case, the SN line reflects at $f$ or $f'$ and  comes back in the
other region $R_{i'}$, eventually joining the edge $g_{i'}(j)$
so that edge is an SN line. Again in that case it has no singularities. The SN line points away from 
$q_i^{(0)}$ in both cases, so towards $q_i^{(1)}$. 
\end{proof}

On the edge $g_i(j)$ 
not concerned by the above lemma, 
let $q_i^{(1)},\ldots , q_i^{(m_{ij})}$ 
denote the singularities in order 
along the interior of $g_i(j)$, starting from the
nearest to $q_i^{(0)}$ (recall that was $s_i(j'')$ or $b,b'$). 
The internal singularities are $4$ or $6$-fold points.

\begin{lemma}
All the segments $q_i^{(u-1)}q_i^{(u)}$ are fold
lines for $1\leq u\leq m_{ij}$. Furthermore, 
for any $1\leq u <m_{ij}$ the singularity $q_i^{(u)}$ 
is a $\six_4$ point with the connecting segments of the edge $g_i$ being its spine. These are oriented in the direction away from $q_i^{(0)}$.
\end{lemma}
\begin{proof}
If any of these segments were SN lines that would create a BPS state. Note that if $q_i^{(0)}$ is not the singular
one of the two points $s_1(j''),s_2(j'')$, then an SN line from $q_i^{(1)}$ to $q_i^{(0)}$ 
would continue to one of the edges $f,f'$ and reflect and hit the corresponding point $s_{i'}(j')$ on the other piece,
that is then a singularity by part 2 of Proposition
\ref{properties}.

If $1\leq u <m_i$ then both segments 
$q_i^{(u-1)}q_i^{(u)}$
and $q_i^{(u)}q_i^{(u+1)}$ are fold edges. 
The only type of singularity with two fold edges 
separated by three sectors is
a $\six _4$ point and these edges are the spine. For the orientation, note that all the segments from $q_i^{(0)}$ 
up to $q_i^{(u+1)}$ are fold edges, and 
$q_i^{(0)}q_i^{(1)}$ is oriented outwards 
from $q_i^{(0)}$. 
By the rule for orientations at $\six_4$
it follows inductively that all the segments 
are oriented the same way. 
\end{proof}

When we do our glue and trim operation, the combination of 
any edge with a folded edge has the same marking as the first edge. So, for the problem of deciding which edges are in the scaffolding, the only indeterminacy is whether the edge 
$q_i^{(m_{ij})}q_i^{(m_{ij}+1)}$ 
is a fold edge or an SN edge. 
This could mean the entire edge $g_i(j)$ if $m_{ij}=0$. 

The following lemma says that the spectral network or fold lines created in the glued and trimmed construction will satisfy the refracting property. 

\begin{lemma}
\label{newrefract}
If the edge $q_i^{(m_{ij})}q_i^{(m_{ij}+1)}$ 
is an SN line, then it points
in the direction towards the endpoint 
$q_i^{(m_{ij}+1)}=t_i(j')$. 
As a corollary, 
if a spectral network line crosses into $R_1$ from anywhere
across the edge $g_1$, it is reflected from $f$ or $f'$ and
creates two SN lines going out from $R_2$ at the 
corresponding point 
in both outward directions. Same with $1,2$ interchanged.
\end{lemma}
\begin{proof}
The point $q_i^{(m_{ij})}$ 
is a singularity so all outgoing non-folded
edges are initial for the spectral network. 
\end{proof}

\section{The new construction}

Let us now proceed with the collapsing operation of
glueing together $R_1$ and $R_2$, and trimming off the
resulting piece. After doing this we obtain the 
new construction $\Znew$. It is again ecarinate, complete,
normal and simply connected. 

Let $g(j)$ denote the edges in $\Znew$ obtained by 
identifying $g_1(j)\subset R_1$ with $g_2(j)\subset R_2$.
Similarly let $s(j)$ (resp. $t(j)$) denote the
points resulting from the identification of 
$s_1(j)$ and $s_2(j)$ (resp. $s_1(j)$ and $s_2(j)$).
The images of $b,b'$ are again denoted $b,b'$. 
We identify $Z-(R_1\cup R_2)$ with $\Znew -\bigcup g(j)$. 

A point $s(j)$ is obtained by gluing together an $8$-fold point
and a $6$-fold point, having removed $4$ sectors from each one.
It follows that $s(j)$ is a $6$-fold point.

If say $t_1(j)$ is a $\four _2$ then $t_2(j)$ is a
$\six_0$ point, indeed the edges 
$g_1(j)$ and
$g_1(j+1)$ must be folded so by Lemma \ref{onesn} 
both edges $g_2(j)$ and
$g_2(j+1)$ are SN lines. 
It follows in this case that $t(j)$ is a $\six_0$ point. 

If $t_i(j)$ are both $6$-fold points then $t(j)$ is an $8$-fold point. 

Suppose that an edge say $g_1(j)$ contains some
singularities $q_1^{(1)},\ldots , q_1^{(m)}$. As we have seen,
$q_1^{(1)}, \ldots , q_1^{(m-1)}$ are $\six_4$ points with
spine along $q_1(j)$. One can see that these become $\six_2$ points in $\Znew$. The segments of $g_1(j)$ that are
folded, become unfolded segments in $\Znew$ since
the opposite edge $g_2(j)$ is an SN line and a folded
segment glued to an unfolded one yields an unfolded segment in
the new construction. 

In order to determine the scaffolding of $\Znew$ it remains
to see what becomes of the last segment 
$q_1^{(m)}q_1^{(m+1)}$,  in 
case $g_1(j)$ contains some singularities, or of the
whole of $g(j)$ in case neither side contains singularities in
the interior. This depends on the singularities at $t_i(j')$
and on the last singular point $q_1^{(m)}$ if it is there,
or else on $s_i(j'')$. 

One must do an analysis of cases. The conclusion is
that the labeling of the resulting segment of $g(j)$ as folded or unfolded, is determined by these singularities.
In some cases the answer is indeterminate at 
$t(j')$ but determined by the other endpoint.
The determinations of fold/unfold coming 
from the two ends of the segment must be the same, 
because we are supposing the existence of some harmonic
map to a building compatible with the scaffolding.

If a segment becomes folded, then it is
oriented outwards from the $8$-fold point $t(j')$ 
(if it is a folded segment then we were not in the case
where one $t_i(j')$ is a $4$-fold point, treated above). 

We have seen in Lemma \ref{newrefract} 
that such a folded segment will
satisfy the refracting condition for reflection of SN lines. 
We have determined the scaffolding of $\Znew$, satisfying the
refracting condition, and with orientations of edges. 

Notice
that the two edges meeting  in $t(j)$ cannot both have singular points in the interior, as that would have meant that 
there was an intersection of fold lines in the 
interior of $R_1$ or $R_2$. 

What remains to be verified is that the new points $s(j)$, $t(j)$, and the images of the $q_i^{(m)}$ if they are there, 
fall into
the standard list of possible singularities; that the
required spectral network lines emanating from these points
exist; and that the orientations of fold lines are compatible
with the allowable orientations at these new singularities.

Discuss first the compatibilities of orientations 
of fold lines. An $8$-fold point $s_i(j)$ 
becomes a $6$-fold point $s(j)$. 
Any
fold lines not on the edges $g(j),g(j+1)$
will stay oriented in the outgoing direction. 
There may be one or two new fold lines on the edges
$g(j),g(j+1)$, but these will now orient inwards towards
$s(j)$. Indeed,
this only happens if $g(j)$ (resp. $g(j+1)$) has no interior
singularities and $t(j')$ is an $8$-fold point. 
See Table \ref{s8table} below for compatibility. 

Consider singularities interior to the edges.
Suppose say $g_1(j)$ contains a 
sequence of singular points with last one $q_1^{(m)}$. 
Let $q^{(m)}$ denote its image in $\Znew$. The segment 
$t(j')q^{(m)}$ might or might  not be folded. 
If it is folded,
it is oriented outward from $t(j')$ 
hence inward towards $q^{(m)}$.
We should check that this is compatible with the singularity
type of $q^{(m)}$. If $q_1^{(m)}$ is a $4$-fold point then 
so is $q^{(m)}$ and this compatibility holds. If our new segment
is folded it means that the previous
segment $q_1^{(m)}t_1(j')$ had to be unfolded. Recall that 
$q_1^{(m-1)}q_1^{(m)}$ is folded and oriented towards 
$q_1^{(m)}$. 
The direction into $R_1$
that is parallel to $f$ or $f'$ is folded, and this fold line
hits one of the edges $f$ or $f'$ at a $\six_4$ point getting
reflected back into $R_2$ and eventually forming a fold line
that will participate in the local picture at the point 
$q^{(m)}$. 

Using these facts there are two possibilities. First,
$q_1^{(m)}$ could be a $\six_3$ point whose fold edge along
$g_1(j)$, the segment $q_1^{(m-1)}q_1^{(m)}$, is oriented
inwards. The new segment
$t(j')q^{(m)}$ is not folded, and the 
resulting point $q^{(m)}$
is a nonsingular $\six_2$ point with compatible orientations of
fold edges (see the middle two lines of Table \ref{q6table}). 

The other possibility is that $q_1^{(m)}$ be a $\six_4$ point.
In this case the new segment $t(j')q^{(m)}$
must be folded, and oriented inwards towards $q^{(m)}$
since $t(j')$ is an $8$-fold point,
whereas $q^{(m-1)}q^{(m)}$ is
unfolded.  The resulting point $q^{(m)}$
is a $\six_4$ point. 
The inward oriented edge of its spine comes from the
inward oriented edge of the spine of $q_1^{(m)}$ reflected into $R_2$. 
See the next-to-last line of Table
\ref{q6table} below for the other 
fold edges and their orientations. 

This completes the discussion 
of compatibility of the 
orientations of new scaffolding edges. 

We need to verify that the new singular types are 
all contained in the list that we are using. 
This will be done by writing
down tables of the
possibilities. The cases of the points $b,b'$
are similar and are left to the reader.

Let us assume that we have a singular point on $R_1$.
Also, for singularities inside edges, we consider edges $g_1(j)$ 
that are parallel to $f$. The other cases are the same by symmetry. 

It will be convenient to establish a convention for speaking
of directions in $R_1,R_2$ at a singularity $s,t,q$. 
Number them as follows: 
$\dxo$ corresponds to the direction from our singularity, towards the interior of $R_i$, not parallel to $f$ or $f'$
(``up'' in the picture). 
Then $\dxa,\dxb, \ldots$ are the directions obtained by
turning clockwise $1,2,\ldots $ sectors. On the
other side, $\dxma, \dxmb, \ldots $ are the directions obtained
by turning counterclockwise that many sectors. These join up:
at a $6$-fold point $\dxc=\dxmc$ while at an $8$-fold 
point $\dxd = \dxmd$. 

The edge $f$ is oriented $\dxa$, and $f'$ is oriented $\dxma$. 

In our tables below we will be listing the folded directions
at singular points. In this case, a notation $\langle i \rangle$
means the edge germ emanating from the singularity in the
specified direction. With this notation, in the orientation of
our scaffolding, the fold edge is said to be oriented outwards.
The same edge oriented inwards towards the singularity
will be denoted $\overline{\langle i \rangle }$. 

We now consider a singularity of the form $s_1(j)$. It
is an $8$-fold point, and the corresponding $s_2(j)$ is
a nonsingular $6$-fold point (either $\six _0$ or $\six_2$).
These glue together to form a point $s(j)\in \Znew$. 
Table \ref{s8table} 
gives the structure of the scaffolding at 
$s(j)$ as a function of the structure at $s_1(j)$.

In order to fill in the table, recall that four sectors 
are removed from the neighborhood of $s_1(j)$, as well as
from the neighborhood of $s_2(j)$; then the remaining two sectors from $s_2(j)$ are glued back in to give the neighborhood of $s$. We make the convention that edge germs at $s(j)$ are numbered starting with the middle edge of the two sectors from $s_2(j)$ being $\dxo$; the two indeterminate lines are $\dxa , \dxma$,
then remaining $\dxb , \dxc , \dxmb$. These latter correspond to the directions $\dxc , \dxd , \dxmc$ respectively at $s_1(j)$.

We include a column in the table to say what is happening at the
point $s_2(j)$. Note that it is a nonsingular $6$-fold
point, hence either $\six_0$ or $\six_2$. If it is $\six_2$ then
the direction of the fold line comes from the direction of the
fold line at $s_1(j)$ that goes in direction either $\dxa$ 
or $\dxma$. This extra column will be most useful in our
third table below.   

We don't include configurations that are obtained by symmetry
(changing $\langle i \rangle$
to $\langle -i \rangle$) from ones that were 
already included, and we don't include configurations
(such as $\eight _1, \dxo$) that are ruled out. 

\begin{table}[h]
\caption{\label{s8table} Structure at $s(j)$}
\begin{tabular}{|c|c|c|c|c|}
\hline
$s_1(j)$   & \mbox{fold edges} & $s_2(j)$ &  $s(j)$     & \mbox{new fold edges}        \\
\hline
$\eight_0$ & \mbox{none} &  $\six_0$ & $\six_0$ & \mbox{none}
\\
\hline
$\eight_1$ & $\dxa_1$ & $\six_2\; \rule{0pt}{12pt}
\overline{\dxa}_2 , \dxmb_2$ & $\six_0$ &  \mbox{none}
\\
\hline
$\eight_1$ & $\dxb_1$ & $\six_0$ & $\six_0$ &  \mbox{none}
\\
\hline
$\eight_1$ & $\dxc_1$ & $\six_0$ & $\six_2$ &  
$\dxb , \overline{\dxma} \rule{0pt}{12pt}$
\\
\hline
$\eight_1$ & $\dxd_1$ & $\six_0$ & $\six_3$ &  $\dxc , 
\overline{\dxa}, \overline{\dxma}\rule{0pt}{12pt}$ 
\\
\hline
$\eight_2$ & $\dxa_1 , \dxmc_1$ & $\six_2\; \rule{0pt}{12pt}
\overline{\dxa}_2 , \dxmb_2$ & $\six_2$ &   
$\dxmb ,\overline{\dxa}\rule{0pt}{12pt}$
\\
\hline
$\eight_2$ & $\dxb_1 , \dxmb_1$ & $\six_0$ & $\six_0$ &  \mbox{none}
\\
\hline
$\eight'_2$ & $\dxa_1 , \dxb_1$ & $\six_2\; \rule{0pt}{12pt}
\overline{\dxa}_2 , \dxmb_2$ & $\six_0$ &  \mbox{none}
\\
\hline
$\eight'_2$ & $\dxb_1 , \dxc_1$ & $\six_0$ & $\six_2$ &  
$\dxb , \overline{\dxma}\rule{0pt}{12pt}$
\\
\hline
$\eight'_2$ & $\dxc_1 , \dxd_1$ & $\six_0$ & $\six_4$ &  
$\dxb , \dxc ,
\overline{\dxa}, \overline{\dxma}\rule{0pt}{12pt}$
\\
\hline
\end{tabular}
\end{table}

Along an edge $g_1(j)$ parallel to $f$ 
suppose given singularities
$q_1^{(1)},\ldots , q_1^{(m)}$. 
We have seen that for $1\leq u <m$
the $q_1^{(u)}$ have to be $\six_4$ points with spine along
$g_1(j)$ resulting in a nonsingular 
$\six_2$ point in $\Znew$ (as shows up in the first lines of the next table). 

Let us consider now the 
configurations for $q_1^{(m)}$ and resulting configurations for $q^{(m)}$ in $\Znew$, shown in Table \ref{q6table}. 
Recall that
in this case, the three sectors of $R_1$ and the three sectors of $R_2$ are cut out, and the remaining ones are glued together. We number the edges at the new point $q^{(m)}$ as follows: the 
edges $\dxb , \dxc$ exterior to $R_1$ keep the same numbers, whereas $\dxb , \dxc$ exterior to $R_2$ 
become respectively $\dxo , \dxma$ (in practice an edge $\dxma$ at $q_1^{(m)}$ reflects becoming $\dxb$ at the nonsingular point $q_2^{(m)}\in R_2$ opposite $q_1^{(m)}$
hence $\dxo$ at $q^{(m)}$); the directions $\dxa$ and $\dxmb$ correspond to the edge $g(j)$. Recall that the direction
$\overline{\dxa}$ 
is by hypothesis folded on $g_1(j)$ so the new edge $\dxa$ is unfolded. If $\dxmb$ is folded at $q_1^{(m)}$ then 
it becomes unfolded at $q^{(m)}$, whereas if it is unfolded then it  can become either.

\begin{table}[h]
\caption{\label{q6table} Structure at $q(m)$}
\begin{tabular}{|c|c|c|c|c|}
\hline
$q_1^{(m)}$   & \mbox{fold edges} & $q_2^{(m)}$ &  $q^{(m)}$     & \mbox{new fold edges}       
\\
\hline
$\six_4$ & $\rule{0pt}{12pt}\overline{\dxa}_1 , 
\overline{\dxma}_1 , 
\dxmb_1 , \dxc_1$ & $\six_2\; \dxma_2, \overline{\dxb_2}$& $\six_2$ &  
$\dxc , \overline{\dxo} \rule{0pt}{12pt}$
\\
\hline
$\six_4$ & $\rule{0pt}{12pt}\overline{\dxa}_1 , \dxma_1 , 
\dxmb_1 , \overline{\dxc}_1$ & 
$\six_2\; \overline{\dxma}_2, \dxb_2$ & $\six_2$ &  
$ \dxo,\overline{\dxc} \rule{0pt}{12pt}$
\\
\hline
$\six_3$ & $\rule{0pt}{12pt}\overline{\dxa}_1 , 
\overline{\dxma}_1 , 
\dxc_1$ &  
$\six_2\; \dxma_2, \overline{\dxb_2}$ & $\six_2$ &  
$\dxc , \overline{\dxo} \rule{0pt}{12pt}$
\\
\hline
$\six_3$ & $\rule{0pt}{12pt}\overline{\dxa}_1 , \dxma_1 , 
\overline{\dxc}_1$ &  
$\six_2\; \overline{\dxma}_2, \dxb_2$ & $\six_2$ &  
$ \dxo,\overline{\dxc} \rule{0pt}{12pt}$
\\
\hline
$\six_4$ & $\rule{0pt}{12pt}\overline{\dxa}_1 , 
\overline{\dxma}_1 , 
\dxb_1 , \dxc_1$ &  
$\six_2\; \dxma_2, \overline{\dxb_2}$ & $\six_4$ &  
$\dxb , \dxc , \overline{\dxo},
\overline{\dxmb} \rule{0pt}{12pt}$
\\
\hline
$\four_2$ & $\rule{0pt}{12pt}\overline{\dxa}_1 , 
\overline{\dxma}_1$ & 
$\six_2\; \dxma_2, \overline{\dxb_2}$ & $\four_2$ &  
$\overline{\dxo} ,\overline{\dxmb}  \rule{0pt}{12pt}$
\\
\hline
\end{tabular}
\end{table}

We now turn to the case of the singular points $t_1(j)$ 
glueing to the nonsingular $t_2(j)$ to yield $t(j)$. 
In this case, two sectors are removed from the neighborhood
of $t_1(j)$, two sectors removed from the neighborhood of $t_2(j)$, and the remaining sectors are put back together.
There are four remaining sectors from $t_2(j)$.  
We make the following 
labeling conventions, with subscripts indicating
sectors coming from neighborhoods of $t_1(j)$ or $t_2(j)$:
$$
\dxb _1 \mapsto \dxc ,\;\;\;\;  \dxc _1 \mapsto \dxd ,
\dxmb _1 \mapsto \dxmc ,  
$$
$$
\dxb _2 \mapsto \dxa ,\;\;\;\;  \dxc _2 \mapsto \dxo ,
\dxmb _2 \mapsto \dxma ,  
$$
and the indeterminate ones
$$
\dxa _1 \mbox{ or } \dxa _2\mapsto \dxb ,\;\;\;\;  
\dxma _1 \mbox{ or } \dxma _2\mapsto \dxmb . 
$$
In this case the structure of $t_2(j)$ is not determined by
that of $t_1(j)$ since it could depend on the singularities
along the adjacent edges, so it is included in the table.
The possibilities are unfolded, $\six_2$ folded in direction 
$\overline{\dxa},\dxmb$ or $\six_2$ folded in 
direction $\overline{\dxma} , \dxb$. 
If for example there is a fold in direction 
$\overline{\dxa} ,\dxmb$
then it came from a fold line in direction $\dxa$ at the
singularity $s_1(j)$ that was reflected on the edge $f'$,
and in this case the edge $g_1(j)$ is unfolded,  in particular
the direction $\dxa$ at $t_1(j)$ is unfolded. Similarly in the other direction. Again we omit cases that can be
obtained by symmetry. 

\begin{table}[h]
\caption{\label{t6table} Structure at $t(j)$}
\begin{tabular}{|c|c|c|c|c|}
\hline
$t_1(j)$   & \mbox{fold edges}   & $t_2(j)$ & $t(j)$     & \mbox{new fold edges}       
\\
\hline
$\six_0$ & \mbox{none} & $\six_0$ & 
$\eight_0, \eight_1,\eight_2$ &  
$\dxb ?, \dxmb ? $
\\
\hline
$\six_0$ & 
\mbox{none} & $\six_2 \;
\rule{0pt}{12pt}\overline{\dxma}_2,\dxb_2 $& 
$\eight_1, \eight '_2$ &  
$\dxa ,\dxb ? $
\\
\hline
$\six_2$ & $\rule{0pt}{12pt}
\overline{\dxa}_1,\dxmb _1$ & $\six_0$ & 
$\eight_1,\eight '_2$ &  
$\dxmc,  \dxmb ? $
\\
\hline
$\six_2$ & $\rule{0pt}{12pt}
\overline{\dxa}_1,\dxmb_1$ & $\six_2 \;
\overline{\dxma}_2,\dxb_2$  & 
$\eight _2$ &  
$\dxmc,  \dxa $
\\
\hline
$\six_2$ & $\rule{0pt}{12pt}
\overline{\dxa}_1,\dxmb_1$ & $\six_2 \;
\overline{\dxma}_2,\dxb_2$  & 
$\eight _2$ &  
$\dxmc,  \dxa $
\\
\hline
$\six_3$ & $\rule{0pt}{12pt}
\overline{\dxa}_1,\overline{\dxma}_1,\dxc_1$ 
& $\six_0$  & 
$\eight _1$ &  
$\dxd $
\\
\hline
$\six_4$ & $\rule{0pt}{12pt}
\overline{\dxa}_1,\overline{\dxma}_1,\dxb_1 , \dxc_1$ 
& $\six_0$  & 
$\eight '_2$ &  
$\dxc, \dxd $
\\
\hline
$\four_2$ & $\rule{0pt}{12pt}
\overline{\dxa}_1,\overline{\dxma}_1$ 
& $\six_0$  & 
$\six _0$ &  
\mbox{none}
\\
\hline
\end{tabular}
\end{table}

In some rows of the table, the answer is not determined by the
information local to $t_i(j)$. In those cases we have
included the various possibilities. Notice that the
marking of edges $\dxb, \dxmb$ will be determined from what
happens in the two previous tables at the adjacent singularities
on these segments. 

\begin{corollary}
The scaffolding of $\Znew$ is well-defined, with orientations
of the fold edges. 
From the tables, the types of local pictures of the
scaffolding for $\Znew$ are in our standard list $\four_2,\six_3,\six_4,\eight_0,\eight_1,\eight _2,\eight'_2$.
The orientations of the fold edges at these singularities are
compatible with the allowable configurations. 
\end{corollary}

\begin{proposition}
\label{snsubset}
Define the refracting 
spectral network of $\Znew$ to be generated
by an initial SN line going in every non-fold direction
from each of the singularities, and closed under collisions as well as refraction upon crossing fold lines of the new scaffolding. Then, any SN-line of this new spectral network, outside of the edges $g(j)$, is contained in
the previous spectral network of $Z$. SN lines along non-folded
segments of the $g(j)$ are with reversed orientation 
with respect to those
of $Z$. Under the assumption that there were no BPS states in 
the refracting spectral network of $Z$, then there are no BPS 
states in the refracting spectral network of $\Znew$. 
\end{proposition}
\begin{proof} {\em (Sketch)}---We verify in each of the
cases contained in the tables, that there are SN lines in $Z -(R_1\cup R_2)$ corresponding to all non-folded outward
directions from singular points. In the case of non-folded
edges that are segments of the $g(j)$, there were
SN lines in the non-folded segments of $g_i(j)$ going
in the opposite direction. Switch the directions
of these, and when these new SN lines 
meet a $\six_2$ point, notice that it came from a singular
point and the refracted directions are among the
directions containing SN lines of $Z$ (or we
continue along the next segment of $g(j)$ to the next $\six_2$
point). 

At an $8$-fold singularity obtained from the third
table, the SN lines in all directions are generated by the
SN lines along the segments of $g_i(j)$ and $g'_i(j)$,
sometimes by using the collision process in $Z$.  

We should check that there are no BPS states in the new
spectral network. Concerning lines not
on the edges $g(j)$ this comes from the inclusion into the
previous spectral network in $Z-(R_1\cup R_2)$, and the existence of SN lines going outwards from any new singularities
as noted above. 

We therefore need to consider new SN lines along segments of $g(j)$. 
For this, let us notice that in Table \ref{s8table},
whenever a singular $\six_3$ or $\six_4$ is created,
the edges going along $g_(j)$ and $g(j+1)$, there
denoted $\dxa$ and $\dxma$, are folded. Furthermore,
whenever a $\six_2$ is created, one of those two edges
is folded so a BPS state between the two adjacent $s(j),s(j+1)$
is not created. 
Similarly, in \ref{q6table} when a $\six_4$ or point is created,
the segment after it on $g(j)$, denoted there 
$\overline{\dxmb}$, is folded.  

This has only been a sketch of proof, a more detailed discussion is needed in order to follow through all possible SN lines that
might start with directions along the edges $g(j)$. 
\end{proof}

\section{Scholium}
\label{scholium}

We now review what has been done (or sketched) above. 
From $X$ we created an initial complete normal ecarinate 
construction $Z^{\rm init}$ 
and we are assuming that this is done following  
Principles \ref{nobps} and \ref{chooseinitial}. 

Thus $Z^{\rm init}$ is 
provided with a fully refracting scaffolding, 
whose associated spectral network has no BPS states. 
Initially 
the scaffolding has only $\eight_1,\eight_2,\eight'_2,\four_2$
singularities, so the endpoints of the post-caustics are
$\eight_1$ points, and the $8$-fold and $4$-fold
points alternate.

The reduction process will consist
of a sequence of reduction steps starting
with $Z^{\rm init}$. Let us denote by $Z$ the construction
obtained after a certain number of steps. Our goal is to describe the next reduction step. 

Our construction $Z$ is again complete, ecarinate, normal, 
and it is provided with a refracting scaffolding with oriented
fold edges.
Our assumptions are as follows:
\newline
---that the refracting spectral network generated 
by the scaffolding has 
no BPS states;
\newline
---and that the types of points in the scaffolding are in the standard list of Definition \ref{singdef}, taking account orientations of edges.   
 
Suppose that there are some remaining fold edges. 
By Corollary \ref{existsfourtwo}, there exists a $\four_2$ point.
The reduction step will be to 
collapse a neighborhood of this $\four_2$ point in a good way. 

Choose a maximal collapsing pair $R_1,R_2$ at this vertex. 
These regions are arranged with singularities and
edges satisfying the properties of 
Proposition \ref{properties} and the subsequent discussion. 

We then glue together $R_1$ and $R_2$, and trim away the
resulting piece (except for the union of edges $g(j)$),
to get a new construction $\Znew$. This is the result of a single
step of our reduction process. 

The main point is to verify
that $\Znew$ is provided with a well-defined refracting scaffolding that still
satisfies the required properties. We have seen that
the configurations at singularities in $Z$ determine
the fold edges at the new points in $\Znew$. Indeed, 
this is the case at points of the form $s(j)$ and 
$q^{(m)}$, and the only indeterminacy at points $t(j)$
is along segments that will be connected either to
points $q^{(m)}$ or $s(j)$ so the fold edges are determined.
We have also assigned orientations to the fold edges.

Specific analysis of each case allows to fill in 
Tables \ref{s8table}, \ref{q6table}, \ref{t6table}. These
show that the new singularities are only of the types
in our standard list. Furthermore, we see here that the
types of singularities are compatible with the assigned orientations of the fold edges in the new scaffolding.

We noted along the way that the fold edges of the new scaffolding have the required refracting effect on spectral network lines. The sketch of proof of 
Proposition \ref{snsubset}
shows how the new spectral network is a subset of the
previous one, apart from the spectral network lines that might have switched directions along the
unfolded segments of the edges $g(j)$, and this spectral
network doesn't have any BPS states. 

This completes the verification that our new construction
$\Znew$ has the required structures and satisfies the 
required properties so it can be used as the starting point in
a next step of the reduction process.  

\section{Further questions} 
\label{sec-further}

We  have described a single step of the
reduction process. The main question will now be to obtain a
convergence statement 
saying that the process ends in
finitely many steps with a construction $Z^{\rm core}$
whose scaffolding has an empty set of edges. Suppose it
does end. The only singularities of $Z^{\rm core}$ are 
$\eight_0$ points of negative curvature

This reduced construction will be the {\em core} 
of the pre-building
that we are conjecturing to exist in \cite{KNPS2}. 
In order to get the pre-building, the steps of  putting back
in the pieces that have been trimmed off, need to be done
as described in \cite{KNPS2}. 

The construction of the core $Z^{\rm core}$ may be seen
as a $2$-dimensional analogue of the Stallings core graph \cite{ParzanchevskiPuder,Stallings}. 
It should be interesting to compare these combinatorics
to the ones of \cite{AKT1,AKT2,IwakiNakanishi}. 

In current work with Fabian Haiden,
we hope to apply this operation 
$$
X\mapsto Z^{\rm init}  \mapsto Z^{\rm core}
$$
in order to generalize the work of Bridgeland and Smith
constructing stability conditions
\cite{BridgelandSmith} from $SL(2)$ to $SL(3)$. 

The core
$Z^{\rm core}$ has a natural flat structure with geometry
modeled on the standard appartment $A$ for $SL(3)$.
This geometric structure carries with it a natural cyclic
$3$-fold spectral covering with ramification at the $8$-fold
singular points. The $8$-fold singular points are 
negatively curved conical points for the flat metric, with
angles of $480^{\circ}$. The flat metric determines
a conformal and hence complex structure; this is
a modification of the complex structure of the
original Riemann surface $X$, and the cyclic covering
corresponds to a cubic differential. This modification looks to 
be a discrete or possibly ``tropical'' analogue of Labourie's
result \cite[Conjectures 1.6, 1.7]{LabourieInv} \cite{Labourie}, 
replacing a general spectral curve by a 
cubic differential.  

The first 
conjecture is that 
a minimization process provides a stability condition
corresponding to the cubic differential, on the category of
sections of the perverse Schober with fiber $A_2^{CY2}$
corresponding to the cyclic spectral covering. 

If we can do that, then a procedure for defining
the stability condition for a general $SL(3)$-spectral curve
$\Sigma$ 
will be to define the categories $D_{\leq \theta}$ of
objects of phase $\leq \theta$, for any $\theta$ by
considering the core $Z^{\rm core}(e^{i\theta}\Sigma )$ and
using the stability condition given by the conjectured minimization process of the previous paragraph. 

These $t$-structures should each 
satisfy the required axioms for taking part in
a stability condition. 
Therefore, taken together for all phases they should define a stability condition.
We are just beginning to work on this program.

\end{document}